\documentclass[11pt]{article}
\usepackage{graphicx}
\usepackage{amsmath,amssymb,amsthm,amsfonts,mathrsfs}
\usepackage{amssymb}
\usepackage{enumerate}
\usepackage[numbers,sort&compress]{natbib}
\usepackage{amsfonts}
\usepackage{a4wide}

\allowdisplaybreaks

\numberwithin{equation}{section}

\newtheorem{theorem}{Theorem}[section]

\newtheorem{corollary}[theorem]{Corollary}
\newtheorem{lemma}[theorem]{Lemma}

\theoremstyle{definition}

\newtheorem{remark}[theorem]{Remark}

\newcommand{\R}{\mathbb{R}}

\newcommand{\bi}{\nabla}

\usepackage{hyperref}
\hypersetup{colorlinks=true,
citecolor=black,
linkcolor=black,
}
\usepackage{amsmath}
\allowdisplaybreaks[4]

\usepackage[notref,notcite]{showkeys}

\begin{document}
\title{Existence and multiplicity of normalized solutions for the quasi-linear Schr\"{o}dinger equations with mixed nonlinearities
\thanks {The research is  supported by Guangxi Natural Science Foundation(2025GXNSFFA069011) and  National Natural Science Foundation of China (12061012, 12461022).}
}

\author{  Qihan  He\thanks{School of Mathematics and Information Science \& Center for Applied Mathematics of Guangxi (Guangxi University), Guangxi University, Nanning, Guangxi, P. R. China. Email: \texttt{heqihan277@163.com}.}
  \ \
and
  \ \ Hao Wang \thanks{School of Mathematics and Information Science, Guangxi University, Nanning, Guangxi, P. R. China.  Email: \texttt{wanghao11226@163.com}.}
}

\date{}

\maketitle
\maketitle

\begin{abstract}
In this paper, we study the existence and multiplicity of the normalized solutions to the following quasi-linear problem
\begin{equation*}
-\Delta u-\Delta(|u|^2)u+\lambda u=|u|^{p-2}u+\tau|u|^{q-2}u, \text{ in }\R^N,~ 1\leq N\leq4,
\end{equation*}
with prescribed mass
$$\int_{\R^N}|u|^2dx=a ,$$
where $\lambda\in\R$ appears as a Lagrange multiplier and the parameters $a,\tau$ are all positive constants. We are concerned about the mass-mixed case $2<q<2+\frac{4}{N}$ and $4+\frac{4}{N}<p<2\cdot2^*$, where $2^*:=\frac{2N}{N-2}$ for $N\geq3$, while $2^*:=\infty$ for $N=1,2$. We show the existence of normalized ground state solution and normalized solution of mountain pass type. Our results can be regarded as a supplement to Lu et al. ( Proc. Edinb. Math. Soc., 2024) and Jeanjean et al. ( arXiv:2501.03845).
\end{abstract}

\vskip 0.12truein

\noindent {\it Keywords:} Quasilinear Schr\"{o}dinger equation; Multiplicity; Normalized solutions

\vskip 0.08truein

\noindent{\it{2020 Mathematics Subject Classification:}}  35A15, 35J62.

\vskip 0.2truein

\section{Introduction}
\qquad In this paper, we study the following time-dependent quasi-linear Schr\"{o}dinger equations
\begin{equation}\label{1.100}
\begin{aligned}
\left\{
  \begin{array}{ll}
    i\partial_t \phi+\Delta\phi+\Delta(|\phi|^2)\phi
+|\phi|^{p-2}\phi+\tau|\phi|^{q-2}\phi=0,\text{ in } \R^+\times\R^N , \\
    \phi(0,x)=\phi_0(x),\text{ in }\R^N,
  \end{array}
\right.
\end{aligned}
\end{equation}
where $N\geq1$, $\phi:\R^N\rightarrow\mathbb{C}$ is a complex valued function. These types of equations have extensive applications across diverse physical disciplines, such as dissipative quantum mechanics, plasma physics, and fluid mechanics. We refer the readers to \cite{CM1,KS,PM} and their references for more information on the related physical backgrounds.

From the perspective of physics and mathematics, a core problem is the existence and dynamics of standing waves of \eqref{1.100}. We call  a solution of the form $\phi(t,x)=e^{-i\lambda t}u(x)$ as a standing wave, where $\lambda\in\R$ is a parameter. Indeed, $\phi(t,x)$ is a solution of \eqref{1.100} if and only if $u(x)$ satisfies the following quasi-linear Schr\"{o}dinger equation
\begin{equation}\label{1.1}
-\Delta u-\Delta(|u|^2)u+\lambda u=|u|^{p-2}u+\tau|u|^{q-2}u, \text{   in }\R^N.
\end{equation}

When seeking a solution to \eqref{1.1}, one option is to treat $\lambda\in\R$ as a fixed constant, which is called as  the fixed frequency problem. In this case, the existence and multiplicity of solutions to \eqref{1.1} have been intensively studied during the past decades (see \cite{AZ,CLZ,CM,FX,LJ,LJ1,LJ2,LJ3,LJ4,LX,LX1,ZX,ZH} and their references therein). Studying the fixed frequency problem  \eqref{1.1}, ones  commonly explore the following energy functional
\begin{equation*}
I_{\lambda}(u)=\frac{1}{2}(|\bi u|_2^2+\lambda|u|_2^2)+\int_{\R^N}u^2|\bi u|^2dx-\frac{1}{p}|u|_p^p-\frac{\tau}{q}|u|_q^q
\end{equation*}
on the space
\begin{equation*}
X:=\{u\in H^1(\R^N)\mid V(u):=\int_{\R^N}u^2|\bi u|^2dx<\infty\}.
\end{equation*}
It is easy  to check that $u$ is a weak solution to \eqref{1.1} if and only if for any $\varphi\in\mathcal{C}_0^\infty(\R^N)$,
$$
\langle I_{\lambda}'(u), \varphi\rangle:=\lim\limits_{t\rightarrow 0^+}\frac{I_{\lambda}(u+t\varphi)-I_{\lambda}(u)}{t}=0.
$$
Unlike semi-linear equations (where the term $\Delta(|u|^2)u$ is absent), finding solutions for the equation \eqref{1.1} is challenging. The functional related to the quasi-linear term $V(u)$ is non-differentiable in the  space $X$ when $N\geq2$. To tackle this, several arguments have been devised. Initially, in \cite{LJ,PM}, solutions of \eqref{1.1} are acquired by minimizing the functional $I_{\lambda}$ on the  set
$$
\{u\in X\mid\int_{\R^N}|u|^{p+1}dx=1\}.
$$
In the proofs of \cite{LJ,PM}, the non-differentiability of $I_{\lambda}$ essentially does not come into play. Alternatively, as demonstrated in \cite{CM,LJ2}, the quasi-linear problem \eqref{1.1} can be transformed into a semi-linear one through variable substitution, and standard variational methods can be applied to solve it. Furthermore, in \cite{LJ3}, the authors proposed a method for more general quasi-linear equations, reducing the solution of \eqref{1.1} to establishing that $I_{\lambda}$ has a global minimizer on a Nehari manifold. Starting from \cite{LX}, the authors developed a perturbation method in a series of papers \cite{LJ1,LJ4,LX,LX1} to study the existence and multiplicity of solutions for a general class of quasi-linear elliptic equations including the aforementioned model.

On the other hand, ones would like to  search for solutions to \eqref{1.1} having a prescribed mass:
\begin{equation}\label{1.2}
\int_{\R^N}|u|^2dx=a.
\end{equation}
In this instance, the objective is to find  a real number $\lambda \in \mathbb{R}$ and a function $u \in H^1(\mathbb{R}^N)$ solving the equations \eqref{1.1} and \eqref{1.2}. Notably, $\lambda \in \mathbb{R}$ serves as a Lagrange multiplier. Importantly, the non-differentiability of $I_{\lambda}$ basically makes no difference in this analysis. From a physical perspective, this approach appears to be of particular significance, often providing profound insights into the dynamic attributes of the static solutions derived from equation \eqref{1.100}. For  this case, the solutions to the equations \eqref{1.1} and \eqref{1.2} correspond to the critical points of the following energy functional
\begin{equation}\label{1.3}
I(u)=\frac{1}{2}|\bi u|_2^2+\int_{\R^N}u^2|\bi u|^2dx-\frac{1}{p}|u|_p^p-\frac{\tau}{q}|u|_q^q
\end{equation}
restricted on the set
\begin{equation}\label{1.4}
\mathcal{S}_a:=\{u\in X\mid |u|_2^2=a\}.
\end{equation}

When $\tau=0$, the equation \eqref{1.1} is changed into the following equation
\begin{equation}\label{11.1}
-\Delta u-\Delta(|u|^2)u+\lambda u=|u|^{p-2}u, \text{   in }\R^N,
\end{equation}
whose constraint energy functional can be defined as
\begin{equation}\label{11.3}
F(u)=\frac{1}{2}|\bi u|_2^2+\int_{\R^N}u^2|\bi u|^2dx-\frac{1}{p}|u|_p^p
\end{equation}
Based on  the well-known Gagliardo-Nirenberg type inequality (refer to \eqref{2.4}), the functional $F$ is bounded from below on $\mathcal{S}_a$ for any $a>0$ if $ p<4+\frac{4}{N}$ and is unbounded from below on $\mathcal{S}_a$ for any $a>0$ if $p>4+\frac{4}{N}$. Therefore,  $4+\frac{4}{N}$ is the mass critical exponent of quasi-linear equation \eqref{11.1}.
For the $L^2$-subcritical case $2<p<4+\frac{4}{N}$,  Colin et al. \cite{CM1} used the method of constraint minimization to  prove the existence and properties of the normalized solutions. In \cite{ZL1}, Zhang, Li and Wang demonstrated that the equation \eqref{11.1} has infinitely many pairs of normalized solutions for $2<p<2+\frac{4}{N}$ by employing a dual approach. Applying a perturbation method, Jeanjean, Luo and Wang proved the existence and multiplicity of normalized solutions for suitable range of mass and $2+\frac{4}{N}<p<4+\frac{4}{N}$ in \cite{JT}. Concerning the $L^2$-critical case where $p = 4 + \frac{4}{N}$, Ye et al. \cite{YH} showed that there is no minimizer of $F|_{\mathcal{S}_a}$ for all $a > 0$. However, in \cite{LH} and \cite{YH}, they proved the existence and asymptotic behavior of normalized solutions to the equations \eqref{1.2} and \eqref{11.1}  for sufficiently large $a > 0$. For the $L^2$-supercritical case where $4+\frac{4}{N}<p<\left\{
                                                \begin{array}{ll}
                                                 \infty,N=1,2,  \\
                                                   2^*,N=3,
                                                   \end{array}
                                                   \right.
$, Li and Zou \cite{LH} applied a perturbation method to demonstrate the existence of positive normalized solution to the equations \eqref{1.2} and \eqref{11.1}  for $1\leq N\leq3$.
Besides this,  Lu and Mao \cite{MA} considered
 the equation \eqref{1.1} with  $2<q<2+\frac{4}{N}<4+\frac{4}{N}<p\left\{
                                                \begin{array}{ll}
                                                 <\infty,N=1,2,  \\
                                                  \leq 2^*,N=3,
                                                   \end{array}
                                                   \right.
$
and proved  the existence of normalized ground state solution and mountain-pass solution by  a perturbation method. Except this,  we have not found any other results  regarding the normalized solutions for the  quasi-linear Schr\"{o}dinger equations with combined nonlinearities.

We want to stress that the constraint energy functional $I$ is well-defined in $X$ when $2<p,q<2\cdot2^*$, but the results mentioned above are only for the case of $p,q<2^*$, and the authors of \cite{MA} has only showed the existence and multiplicity of normalized solutions to \eqref{1.1} for the mass-mixed case with $p,q\leq 2^*$.
Recently, Jeanjean, Zhang and Zhong studied  the existence and asymptotic behavior of positive  normalized ground state solution to the equations \eqref{1.2} and \eqref{11.1} with $N\geq1$ and  $4+\frac{4}{N}<p<\left\{
                                                \begin{array}{ll}
                                                 \infty,N=1,2,  \\
                                                   2\cdot2^*,N\geq3,
                                                   \end{array}
                                                   \right.
$ and  demonstrated the existence of positive  normalized solutions to the equations \eqref{1.2} and \eqref{11.1} for all mass $a>0$ when $1\leq N\leq4$ in \cite{JL}. Additionally, for dimensions $N\geq5$, they established that there exists  a precise threshold $a_0$ such that a normalized ground state solution exists if and only if $a\in(0,a_0]$. Therefore, we find that there  remains at least a research gap in the mass-mixed case for  $2^*<p< 2\cdot 2^*$. Motivated by \cite{JL,MA,SN}, we want to  consider whether there are normalized solutions to \eqref{1.1} for  $2<q<2+\frac{4}{N}<4+\frac{4}{N}<p<2\cdot2^*$.


Inspired by \cite{GFG,ZL}, we introduce the following Pohozaev manifold
\begin{equation}\label{1.6}
\mathcal{P}:=\{u\in X\setminus\{0\}\mid P(u)=0\},
\end{equation}
where
\begin{equation}\label{1.7}
P(u):=|\bi u|_2^2+(N+2)V(u)-\frac{N(p-2)}{2p}|u|_p^p-\frac{\tau N(q-2)}{2q}|u|_q^q.
\end{equation}
Rather than directly studying the constraint $\mathcal{P}\cap \mathcal{S}_a$ as in \cite{ZL}, and motivated by the work \cite{JL}, we will study a relaxed constraint. For $a>0$, we introduce the sets
\begin{equation}\label{1.10}
 D_a:=\{u\in X\mid |u|_2^2\leq a\}~~\hbox{and}~~\mathcal{P}_a:=\mathcal{P}\cap D_a.
\end{equation}
Then for $u\in D_a \setminus\{0\}$ and $t\in\R^+$, we set
$$t\star u(x):=t^{\frac{N}{2}}u(tx)$$
and the fiber map $t\mapsto \Psi_u(t):=I(t\star u)$, where
\begin{equation}\label{1.8}
I(t\star u)=\frac{1}{2}t^2|\bi u|_2^2+t^{N+2}V(u)-\frac{1}{p}t^{\frac{N(p-2)}{2}}|u|_p^p-\frac{\tau}{q}t^{\frac{N(q-2)}{2}}|u|_q^q.
\end{equation}
Therefore, we have
\begin{equation}\label{1.9}
\begin{aligned}
\Psi'_u(t)&=t|\bi u|_2^2+(N+2)t^{N+1}V(u)-\frac{N(p-2)}{2p}t^{\frac{N(p-2)}{2}-1}|u|_p^p-\frac{\tau N(q-2)}{2q}t^{\frac{N(q-2)}{2}-1}|u|_q^q\\
&=\frac{1}{t}P(t\star u)
\end{aligned}
\end{equation}
and
\begin{equation*}
\begin{aligned}
\Psi''_u(t)=&|\bi u|_2^2+(N+2)(N+1)t^N V(u)-\frac{N(p-2)[N(p-2)-2]}{4p}t^{\frac{N(p-2)}{2}-2}|u|_p^p\\
&-\frac{N(q-2)[N(q-2)-2]}{4q}t^{\frac{N(p-2)}{2}-2}|u|_q^q.
\end{aligned}
\end{equation*}

It is easy to see that
$\mathcal{P}_a$ can be divided into the disjoint union $\mathcal{P}_a=\mathcal{P}_a^+\cup\mathcal{P}_a^0\cup\mathcal{P}_a^-$, where
$$\mathcal{P}_a^+:=\{u\in D_a\setminus\{0\}\mid \Psi'_u(1)=0, \Psi''_u(1)>0\},$$
$$\mathcal{P}_a^0:=\{u\in D_a\setminus\{0\}\mid \Psi'_u(1)=0, \Psi''_u(1)=0\}$$
and
$$\mathcal{P}_a^-:=\{u\in D_a\setminus\{0\}\mid \Psi'_u(1)=0, \Psi''_u(1)<0\}.$$
We will show that  $\mathcal{P}_a^0=\emptyset, \mathcal{P}_a^+\neq \emptyset$ and $\mathcal{P}_a^-\neq \emptyset$ later (see  Lemmas \ref{lemma3.2} and \ref{lemma3.3}). Since $I$ is bounded from below on $\mathcal{P}_a$,   we can define
\begin{equation}\label{1.11}
M_a^+:=\inf_{u\in \mathcal{P}_a^+}I(u),~ M_a^-:=\inf_{u\in \mathcal{P}_a^-}I(u).
\end{equation}

In  Lemmas \ref{lemma3.800} and   \ref{lemma3.8}, we will prove that if  there exist $u^+\in \mathcal{P}_a^+$ such that $I(u^+)=M_a^+$ ($u^-\in \mathcal{P}_a^-$ such that $I(u^-)=M_a^-$), then we can deduce that $u^+ (u^-)$ is a critical point of $I|_{D_a}$.

Thus, we firstly need to find the minimizers $u^\pm\in X\setminus\{0\}$ of $I$ constrained to $D_a$ at the level $M_a^\pm:=\inf\limits_{u\in\mathcal{P}_a^\pm}I(u)$. Next, we only need to prove $u^\pm\in \mathcal{S}_a$ to get the existence of the normalized ground state solution  and the normalized solution of
mountain pass type for the initial problem.

Our results can be stated as follows:
\begin{theorem}\label{Th1}
Suppose that $1\leq N \leq 4$, $2<q<2+\frac{4}{N}<4+\frac{4}{N}<p<2\cdot2^*$, and let $a$ and $\tau$ be positive and satisfy the following inequality:
\begin{equation}\label{1.12}
\begin{aligned}
&\left(\tau a^{\frac{4N-q(N-2)}{2(N+2)}}\right)^{p\gamma_p-(N+2)}\left( a^{\frac{4N-p(N-2)}{2(N+2)}}\right)^{N+2-q\gamma_q}\\
&\quad<\left(\frac{p(N+2-q\gamma_q)}{C_{N,p}(p\gamma_p-q\gamma_q)}\right)^{N+2-q\gamma_q}\left(\frac{q(p\gamma_p-N-2)}{C_{N,q}(p\gamma_p-q\gamma_q)}\right)^{p\gamma_p-(N+2)},
\end{aligned}
\end{equation}
where $\gamma_p, \gamma_q, C_{N,p}, C_{N,q}$ are defined by \eqref{2.5}. Then there exist $u^+\in \mathcal{P}_a^+$ such that $I(u^+)=M_a^+$ and $u^-\in \mathcal{P}_a^-$ such that $I(u^-)=M_a^-$.

Moreover, we have the following facts:

(1)~~The functions $u^{\pm}$ are radially symmetric;

(2)~~There exist $\lambda^{\pm}>0$ such that $(\lambda^+, u^+)$ and $(\lambda^-, u^-)$ are solutions to the equations \eqref{1.1} and \eqref{1.4};

(3)~~$u^+$ constitutes a normalized ground state solution to the equations \eqref{1.1} and \eqref{1.4};

(4)~~$u^-$ arises as a normalized solution of mountain pass type of \eqref{1.1}.

\end{theorem}

\begin{remark}
Previous work established existence results for dimensions $1 \leq N \leq 3$: Li and Zou handled the case $4 + \frac{4}{N} < p < 2^*$ in  \cite{LH}, while Lu and Mao \cite{MA} addressed $2 < q < 2 + \frac{4}{N} < 4 + \frac{4}{N} < p \leq 2^*$, both using perturbation methods. Jeanjean, Zhang, and Zhong \cite{JL} later extended these results to higher dimensions $N = 4$ and $N \geq 5$ for the full range $4 + \frac{4}{N} < p < 2 \cdot 2^*$, removing the previous $N \leq 3$ limitation. They also showed that for $N \geq 5$, solutions exist only when the mass $a$ is below an explicit value $a_0 > 0$. This paper studies normalized solutions for the mixed-mass case in dimensions $1 \leq N \leq 4$ with nonlinear exponents $2 < q < 2 + \frac{4}{N} < 4 + \frac{4}{N} < p < 2 \cdot 2^*$. We prove that both equations \eqref{1.1} and \eqref{1.4} have normalized ground state solutions and mountain-pass type solutions. Our main challenge is to verify that $u^+$ is a weak solution, which we achieve by establishing the existence of the Lagrange multiplier $\lambda^+$ through constrained minimization combined with the implicit function theorem.
\end{remark}

\begin{remark}
The  results of Theorem \ref{Th1} need to be restricted under the  condition $1\leq N\leq 4$, since Lemma \ref{lemma3.11} holds only for the case of $1\leq N\leq 4$, i.e., we are unaware of the existence results for the nonnegative nontrivial solutions to the quasi-linear equation
\begin{equation}\label{6.30.1}
-\Delta u - \Delta(|u|^2)u = |u|^{p-2}u + \tau |u|^{q-2}u \quad \text{in} \quad \mathbb{R}^N, N\geq 5.
\end{equation}
If ones can show that \eqref{6.30.1} has no nonnegative nontrivial solutions or \eqref{6.30.1} has no nonnegative nontrivial solutions with a small $L^2-$ norm, which implies that our Lemma \ref{lemma3.11} is also right for $N\geq5$ or $N\geq 5$ and $a$ suitable small, then our results are also right for $N\geq5$ or $N\geq 5$ and $a$ suitable small.

\end{remark}

Before closing the  introduction, we want to introduce our research ideas. We study a combined nonlinear problem, whose difficulty is that we want to find two solutions: a local minimum solution and a mountain-pass type solution. To ensure the existence of these two solutions (especially the local minimum solution), we require the mass parameter
$a$  to satisfy specific conditions. After that, the mass-mixed case makes that the constraint energy functional has a concave-convex structure.
In the process  of finding normalized solution of mountain pass type,  our methods are basically the same as that in \cite{JL}.
When proving the existence of the  local minimum solution (see Lemma \ref{lemma3.800}), the methods we adopt (such as constructing paths, the implicit function theorem, etc.) are different from the standard methods used in the  previous studies on purely mass subcritical or mass supercritical cases to prove that the attaining function is a weak solution. The key point  is that through these methods, we successfully convert the properties of the minimizer into the existence of a weak solution.

The paper is organized as follows. In Section 2, we collect some preliminary results,  which are crucial for our proof.  We discuss the existence of the normalized  ground state solution and the normalized solution of mountain pass type in Section 3.

\section{Preliminaries}

In this Section, we provide some useful preliminaries. We first introduce some notations:
we write $|u|_p^p:=\int_{\R^N}|u|^pdx$ with $1\leq p<+\infty$ and let  $H_{rad}^1(\R^N)$ denote  the subspace, consisting  of radially symmetric functions of  $H^1(\R^N)$. Moreover, when the parameter $\tau$ is determined, we will  omit the subscript $\tau$ of $I_{\tau}, \mathcal{P}_{\tau}, \mathcal{P}_{a,\tau}$, and   write them as  $I, \mathcal{P}, \mathcal{P}_{a}$.

\begin{lemma}
A Gagliardo-Nirenberg type inequality. We recall the following Gagliardo-Nirenberg type inequality: There is some positive constant $C_{N,t}$ such that
\begin{equation}\label{2.4}
\int_{\R^N}|u|^tdx\leq C_{N,t}\left(\int_{\R^N}|u|^2dx\right)^{\frac{4N-t(N-2)}{2(N+2)}}\left(V(u)\right)^{\frac{N(t-2)}{2(N+2)}},~~ 2<t<2\cdot2^*.
\end{equation}
\end{lemma}
\begin{proof}
The proof of \eqref{2.4} can be found in \cite[Lemma 4.2]{CM1} or \cite[Lemma 4.14]{GF}.
\end{proof}

For the convenience of calculation, we denote
\begin{equation}\label{2.5}
\gamma_t:=\frac{N(t-2)}{2t}, \text{ where }t\in (2,2\cdot2^*).
\end{equation}

\section{The proof of Theorem \ref{Th1}}

In this Section, we prove Theorem \ref{Th1}. The proof is divided into three subsections.

\subsection{Properties of $\mathcal{P}_a$}
 For any $u\in D_a\setminus\{0\}$, using the Gagliardo-Nirenberg inequality \eqref{2.4}, we deduce that
\begin{equation}\label{3.1}
\begin{aligned}
I(u)&=\frac{1}{2}|\bi u|_2^2+V(u)-\frac{1}{p}|u|_p^p-\frac{\tau}{q}|u|_q^q\\
&\geq V(u)-\frac{1}{p}|u|_p^p-\frac{\tau}{q}|u|_q^q\\
&\geq V(u)-\frac{C_{N,p}}{p}a^{\frac{4N-p(N-2)}{2(N+2)}}\left(V(u)\right)^{\frac{p\gamma_p}{N+2}}
-\frac{\tau C_{N,q}}{q}a^{\frac{4N-q(N-2)}{2(N+2)}}\left(V(u)\right)^{\frac{q\gamma_q}{N+2}}.
\end{aligned}
\end{equation}
Therefore, it is natural to study the function $g:[0, +\infty) \rightarrow\R$,
$$g(s):=s-\frac{C_{N,p}}{p}a^{\frac{4N-p(N-2)}{2(N+2)}}s^{\frac{p\gamma_p}{N+2}}
-\frac{\tau C_{N,q}}{q}a^{\frac{4N-q(N-2)}{2(N+2)}}s^{\frac{q\gamma_q}{N+2}} $$
to understand the geometry structure of the functional $I|_{D_a\setminus\{0\}}$.
Because of $\tau>0$ and $\frac{q\gamma_q}{N+2}<1<\frac{p\gamma_p}{N+2}$, we deduce that
$$ g(0^+)=0^- \text{ and } g(+\infty)=-\infty. $$

\begin{lemma}\label{lemma3.1}
If $a$ and $\tau$ are positive and satisfy the inequality \eqref{1.12}, then the function $g(s)$ has exactly two critical points, one is local minimum point  at negative level and the other is global maximum  point at positive level. Meanwhile, there exist $0<R_0<R_1$, both depending on $a$ and $\tau$, such that $g(R_0)=0=g(R_1)$ and $g(s)>0$ if and only if $s\in(R_0,R_1)$.
\end{lemma}
\begin{proof}
We have $g(s)>0$ if and only if
$$
f(s)>\frac{\tau C_{N,q}}{q}a^{\frac{4N-q(N-2)}{2(N+2)}}~,\text{where } f(s):=s^{1-\frac{q\gamma_q}{N+2}}-\frac{ C_{N,p}}{p}a^{\frac{4N-p(N-2)}{2(N+2)}}s^{\frac{p\gamma_p-q\gamma_q}{N+2}}.
$$
By direct calculations, we can see that $f(x)$ attains its maximum
$$
f(\bar{s})=\left(\frac{p(N+2-q\gamma_q)}{C_{N,p}(p\gamma_p-q\gamma_q)}\right)^{\frac{N+2-q\gamma_q}{p\gamma_p-(N+2)}}
\frac{p\gamma_p-(N+2)}{p\gamma_p-q\gamma_q}a^{-\frac{4N-p(N-2)}{2(N+2)}\frac{N+2-q\gamma_q}{p\gamma_p-(N+2)}},
$$
where
$$
\bar{s}:=\left(\frac{p(N+2-q\gamma_q)}{C_{N,p}(p\gamma_p-q\gamma_q)}\right)^{\frac{N+2}{p\gamma_p-(N+2)}}a^{\frac{p(N-2)-4N}{Np-4N-4}}.
$$
%
If $a$ and $\tau$ are positive and satisfy the inequality \eqref{1.12}, then can have that $g(\bar{s})>0$, which, combining $g(0)=0,~g(0^+)=0^-$ and $\lim\limits_{s\to +\infty}g(s)=-\infty$, implies that $g(s)$ has at least a local minimum point $s_0$  and at least a global maximum point $s_1$, and $g(s_0)<0$ and $g(s_1)>0$. It is easy to see that $g^\prime(s_0)=g^\prime(s_1)=0$. So $g(s)$ has at least  two critical points $s_0$ and $s_1$

On the other hand, we can see that  $g'(s)=0$ is equivalent to
\begin{equation}\label{5.281}
s^{1-\frac{q\gamma_q}{N+2}}-\frac{ \gamma_pC_{N,p}}{N+2}a^{\frac{4N-p(N-2)}{2(N+2)}}s^{\frac{p\gamma_p-q\gamma_q}{N+2}}=\frac{\tau \gamma_qC_{N,q}}{N+2}a^{\frac{4N-q(N-2)}{2(N+2)}}.
\end{equation}
Letting $\varphi(s):=s^{1-\frac{q\gamma_q}{N+2}}-\frac{ \gamma_pC_{N,p}}{N+2}a^{\frac{4N-p(N-2)}{2(N+2)}}s^{\frac{p\gamma_p-q\gamma_q}{N+2}}$, direct computation tells us that there exists a $s_3$ such that $\varphi(s)$ increases strictly in $(0, s_3)$ and decreases strictly in $(s_3,+\infty)$, which implies \eqref{5.281} has at most two solutions. Therefore, $g^\prime(s)=0$ has at most two solutions, which, together with $g^\prime(s_0)=g^\prime(s_1)=0$, tells us that  $g(s)$ has only two critical points $s_0$ and $s_1$.
Using the facts that $g(0)=0, g(0^+)=0^-$ and $\lim\limits_{s\to +\infty}g(s)=-\infty$, $g(s)$ has only two critical points $s_0$ and $s_1$ and $g(s_1)>0$, we can deduce that there exists $0<R_0<R_1$ such that $g(R_0)=g(R_1)=0$ and $g(s)>0$ for $s\in (R_0, R_1)$.
\end{proof}

Recalling the decomposition of $\mathcal{P}_a=\mathcal{P}_a^+\cup\mathcal{P}_a^0\cup\mathcal{P}_a^-$, we first prove  the following Lemma.
\begin{lemma}\label{lemma3.2}
If $a$ and $\tau$ are positive and satisfy the inequality \eqref{1.12}, then $\mathcal{P}_a^0=\emptyset$.
\end{lemma}
\begin{proof}
Suppose that there exists $u\in\mathcal{P}_a^0$, then we have $P(u)=0$ and $\Psi''_u(1)=0$, namely,
\begin{equation}\label{3.2}
|\bi u|_2^2+(N+2)V(u)-\gamma_p|u|_p^p-\tau\gamma_q|u|_q^q=0,
\end{equation}
\begin{equation}\label{3.3}
|\bi u|_2^2+(N+2)(N+1)V(u)-\gamma_p(p\gamma_p-1)|u|_p^p-\tau\gamma_q(q\gamma_q-1)|u|_q^q=0.
\end{equation}
Using \eqref{3.2} and \eqref{3.3}, we obtain that
\begin{equation}\label{3.4}
(p\gamma_p-2)|\bi u|_2^2+(N+2)(p\gamma_p-(N+2))V(u)-\tau\gamma_q(p\gamma_p-q\gamma_q)|u|_q^q=0.
\end{equation}
Since $p\gamma_p>N+2$ and $p\gamma_p>q\gamma_q$ for $2<q<2+\frac{4}{N}$ and $4+\frac{4}{N}<p<2\cdot2^*$, by using \eqref{3.4}, we have that
$$ (N+2)(p\gamma_p-(N+2))V(u)\leq\tau\gamma_q(p\gamma_p-q\gamma_q)|u|_q^q, $$
which, combined with \eqref{2.4} and \eqref{2.5}, implies that
\begin{equation}\label{3.5}
V(u)\leq \frac{\tau\gamma_q C_{N,q}(p\gamma_p-q\gamma_q)}{(N+2)(p\gamma_p-(N+2))}a^{\frac{4N-q(N-2)}{2(N+2)}}(V(u))^{\frac{q\gamma_q}{N+2}}.
\end{equation}
Similarly, we also have
$$V(u)\leq \frac{\gamma_p C_{N,p}(p\gamma_p-q\gamma_q)}{(N+2)(N+2-q\gamma_q)}a^{\frac{4N-p(N-2)}{2(N+2)}}(V(u))^{\frac{p\gamma_p}{N+2}},$$
which, together with \eqref{3.5}, implies that
\begin{equation}\label{3.6}
\begin{aligned}
&\left(\frac{(N+2)(N+2-q\gamma_q)}{\gamma_pC_{N,p}(p\gamma_p-q\gamma_q)}\right)^{N+2-q\gamma_q}
\left(\frac{(N+2)(p\gamma_p-N-2)}{\gamma_qC_{N,q}(p\gamma_p-q\gamma_q)}\right)^{p\gamma_p-(N+2)}\\
&\quad\leq\left(\tau a^{\frac{4N-q(N-2)}{2(N+2)}}\right)^{p\gamma_p-(N+2)}\left( a^{\frac{4N-p(N-2)}{2(N+2)}}\right)^{N+2-q\gamma_q}.
\end{aligned}
\end{equation}
By a similar argument of Lemma 5.2 of \cite{SN}, we can obtain that this is contradiction with \eqref{1.12}. Therefore, we have that $\mathcal{P}_a^0=\emptyset$.
\end{proof}
\begin{lemma}\label{lemma3.3}
For any $0\neq u\in D_a$, if $a$ and $\tau$ are positive and satisfy the inequality \eqref{1.12}, then the function $\Psi_u(t)$ has exactly two critical points $s_u<t_u\in\R^+$ and two zero points $c_u<d_u\in\R^+$ such that $s_u$ is a local minimum point for $\Psi_u(t)$ and $t_u$ is a global maximum point for $\Psi_u(t)$, with $s_u<c_u<t_u<d_u$. Moreover,
\begin{itemize}
  \item [(i)] $V(t\star u)\leq R_0$ for every $t\leq c_u$, $s_u\star u\in \mathcal{P}_a^+$ and
$$ I(s_u\star u)=\min\{I(t\star u)\mid t\in\R^+\text{ and }V(t\star u)<R_0\}<0.$$
  \item [(ii)] $t_u\star u\in \mathcal{P}_a^-$ and
$$ I(t_u\star u)=\max\{I(t\star u)\mid t\in\R^+\}>0.$$
  \item [(iii)] The maps $u\mapsto s_u\in\R^+$ and $u\mapsto t_u\in\R^+$ are of class $\mathcal{C}^1$.
\end{itemize}
\end{lemma}
\begin{proof}
Let $u\in D_a\setminus\{0\}$. We first prove that $\Psi_u(t)$ has at least two critical points.\\
By \eqref{3.1}, we obtain that
$$\Psi_u(t)=I(t\star u)\geq g(V(t\star u))=g(t^{N+2}V(u)).$$
Since $\Psi_u(t)$ is $\mathcal{C}^2$, positive on $\left( \left( \frac{R_0}{V(u)} \right)^{\frac{1}{N+2}}, \left( \frac{R_1}{V(u)} \right)^{\frac{1}{N+2}} \right)$,
and satisfies $\Psi_u(0^+)=0^-$ and $\lim\limits_{t\rightarrow+\infty}\Psi_u(t)=-\infty$,
it has at least two critical points:
a local minimum point $s_u$ in $\left(0,\left( \frac{R_0}{V(u)} \right)^{\frac{1}{N+2}} \right)$ with $\Psi_u(s_u)<0$,
and a local maximum point $t_u>s_u$ with $\Psi_u(t_u)>0$.

Next we show there are no other critical points.\\
By \eqref{1.9}, $\Psi'_u(t)=0$ if and only if
\begin{equation}\label{3.7}
f(t)=\frac{\tau N(q-2)}{2q}|u|_q^q,
\end{equation}
where
\[ f(t)=|\nabla u|_2^2 t^{2-\frac{N(q-2)}{2}}+(N+2)V(u)t^{N+2-\frac{N(q-2)}{2}}-\frac{N(p-2)}{2p}|u|_p^pt^{\frac{N(p-q)}{2}}.\]
Direct calculation shows $f(t)$ has exactly one maximum point for $t>0$. Thus \eqref{3.7} has at most two solutions.
Therefore $\Psi_u(t)$ has exactly two critical points.
With \eqref{1.9}, this implies $s_u\star u,t_u\star u\in\mathcal{P}_a$
and shows $t_u$ is the global maximum point of $\Psi_u(t)$.

In addition, we have $V(t \star u) \leq R_0$ for all $t \leq c_u$.
If this is false, there exists $t \leq c_u$ such that $V(t \star u) > R_0$. Since $I(t \star u) \leq 0$ for $t \leq c_u$, we get $V(t \star u) \geq R_1$. From $R_0 < V(t_u \star u) < R_1$ and the scaling $V(t \star u) = t^{N+2} V(u)$, we see that $t^{N+2} V(u) \geq R_1 > V(t_u \star u) = t_u^{N+2} V(u)$. This implies $t > t_u$ (since $V(u) > 0$), contradicting $t \leq c_u < t_u$. By the minimality property, $\Psi''_{s_u\star u}(1)=\Psi''_{u}(s_u) \geq 0$. As $\mathcal{P}_a^0 = \emptyset$, we have $s_u \star u \in \mathcal{P}_a^+$.
Similarly, $t_u \star u \in \mathcal{P}_a^-$.

Finally, using the implicit function Theorem to the $\mathcal{C}^1$ function $g(t,u):\R^+\times D_a\mapsto\R$ defined by
$$ g_u(t)=g(t,u)=\Psi'_u(t),$$
then we obtain that $u\mapsto s_u\in \R^+$ is of class $\mathcal{C}^1$ due to $g_u(s_u)=0$ and $\partial_s g_u(s_u)=\Psi''_u(s_u)>0$. Similarly, we can also deduce that $u\mapsto t_u\in \R^+$ is of class $\mathcal{C}^1$.

This completes the proof of the Lemma \ref{lemma3.3}.
\end{proof}

\subsection{Properties of $M^{\pm}_a$}
For $k>0$, we define
$$ A_k:=\{u\in D_a\setminus\{0\}\mid V(u)<k\}\text{ and }m(a,\tau):=\inf_{u\in A_{R_0}}I(u),$$
then using the Lemma \ref{lemma3.3}, we have the following Corollary.
\begin{corollary}\label{corollary3.4}
The set $\mathcal{P}_a^+$ is contained in $A_{R_0}$ and
$$ \sup_{u\in\mathcal{P}_a^+}I(u)\leq 0\leq\inf_{u\in\mathcal{P}_a^-}I(u). $$
\end{corollary}
\begin{lemma}\label{lemma3.5}
Let $a$ and $\tau$ be positive and satisfy the inequality \eqref{1.12}, we have that
\begin{equation}\label{3.8}
m(a,\tau)=\inf_{u\in\mathcal{P}_a}I(u)=\inf_{u\in\mathcal{P}_a^+}I(u)=M_a^+,\text{ with } m(a,\tau)\in(-\infty,0)
\end{equation}
and
\begin{equation}\label{3.9}
M_a^+:=\inf_{u\in\mathcal{P}_a^+}I(u)=\inf_{u\in D_a\setminus\{0\}}\min_{0<t\leq s_u}I(t\star u)
\end{equation}
and
\begin{equation}\label{3.10}
M_a^-:=\inf_{u\in\mathcal{P}_a^-}I(u)=\inf_{u\in D_a\setminus\{0\}}\max_{s_u<t\leq t_u}I(t\star u)>0.
\end{equation}
\end{lemma}
\begin{proof}

Firstly, for any $u\in A_{R_0}$, \eqref{3.1} yields
\[ I(u)\geq g(V(u))\geq \min_{s\in(0,R_0)}g(s)>-\infty, \]
so $m(a,\tau)>-\infty$. Also,
\[ V(t\star u)<R_0 ~ \text{and} ~ I(t\star u)<0, ~\text{as} ~ t\to0, \]
which implies that $m(a,\tau)<0$.

By Corollary \ref{corollary3.4}, we obtain that $\mathcal{P}_a^+\subset A_{R_0}$. So
\begin{equation}\label{3.11}
m(a,\tau)=\inf_{u\in A_{R_0}}I(u)\leq\inf_{u\in\mathcal{P}_a^+}I(u).
\end{equation}
Lemma \ref{lemma3.3} tells us  that if $u\in A_{R_0}$, then $s_u\star u\in\mathcal{P}_a^+\subset A_{R_0}$. Thus
\[ I(s_u\star u)=\min_{t>0, V(t\star u)<R_0} I(t\star u) \leq I(u), \]
which means that
\[ \inf_{u\in\mathcal{P}_a^+}I(u)\leq \inf_{u\in A_{R_0}}I(u)= m(a,\tau). \]
Together with  \eqref{3.11}, we obtain that
\[ m(a,\tau)=\inf_{u\in A_{R_0}}I(u)=\inf_{u\in\mathcal{P}_a^+}I(u)=M_a^+. \]

In addition, applying  Lemma \ref{lemma3.1} and Corollary \ref{corollary3.4}, we have that $I(u)>0$ on $\mathcal{P}_a^-$. Thus
\[ \inf_{u\in\mathcal{P}_a}I(u)=\inf_{u\in\mathcal{P}_a^+}I(u). \]
Therefore \eqref{3.8} holds.

It follows from Lemma \ref{lemma3.3}   that for every $u\in D_a\setminus\{0\}$:
\begin{itemize}
\item There is a unique local minimum $s_u>0$ with $s_u\star u\in\mathcal{P}_a^+$ and $\Psi_u$ strictly decreasing on $(0,s_u]$
\item There is a unique global maximum $t_u$ with $t_u\star u\in\mathcal{P}_a^-$ and $\Psi_u$ strictly increasing on $(s_u,t_u]$
\end{itemize}
Therefore, for any $u\in\mathcal{P}_a^+$, we have that $s_u=1$. So
\[ I(u)=I(s_u\star u)=\min_{0<t\leq s_u}I(t\star u)\geq\inf_{v\in D_a\setminus\{0\}}\min_{0<t\leq s_v}I(t\star v), \]
which implies that
\begin{equation}\label{3.12}
\inf_{u\in\mathcal{P}_a^+}I(u)\geq\inf_{u\in D_a\setminus\{0\}}\min_{0<t\leq s_u}I(t\star u).
\end{equation}
For any $u\in D_a\setminus\{0\}$, we deduce that $s_u\star u\in\mathcal{P}_a^+$. Therefore,
\[ \min_{0<t\leq s_u}I(t\star u)=I(s_u\star u)\geq\inf_{v\in\mathcal{P}_a^+}I(v), \]
and thus
\[ \inf_{u\in D_a\setminus\{0\}}\min_{0<t\leq s_u}I(t\star u)\geq \inf_{u\in\mathcal{P}_a^+}I(u), \]
which, together with \eqref{3.12}, implies that
\[ \inf_{u\in\mathcal{P}_a^+}I(u)=\inf_{u\in D_a\setminus\{0\}}\min_{0<t\leq s_u}I(t\star u). \]

Using a similar method, one can show that
\[ \inf_{u\in\mathcal{P}_a^-}I(u)=\inf_{u\in D_a\setminus\{0\}}\max_{s_u<t\leq t_u}I(t\star u). \]
Moreover, let $s_{\max}$ denote the strict maximum of $g$ at positive level (see Lemma~\ref{lemma3.1}). For each $u \in \mathcal{P}_a^-$, there exists $t_u > 0$ such that $V(t_u \star u) = s_{\max}$. Since $u \in \mathcal{P}_a^-$, Lemma~\ref{lemma3.3} implies that $1$ is the unique strict maximum point of $\Psi_u$. Therefore,
\[
I(u) = \Psi_u(1) \geq \Psi_u(t_u) = I(t_u \star u) \geq g(t_u \star u) = g(s_{\max}) > 0.
\]
As $u \in \mathcal{P}_a^-$ was arbitrary, we conclude that
\[
\inf_{u \in \mathcal{P}_a^-} I(u) \geq  g(s_{\max}) > 0.
\]

This completes the proof.
\end{proof}

Let $X^{rad}$ be the radial subspace of $X$, i.e.,
$$X^{rad}:=\{u\in X\mid u(x)=u(|x|),x\in\R^N\}.$$
Define
\begin{equation}\label{3.13}
D_a^{rad}:=D_a\cap X^{rad}\text{ and } M_{r,a}^\pm:=\inf\limits_{u\in\mathcal{P}_a^\pm\cap X^{rad}}I(u).
\end{equation}
Similar to Lemma \ref{lemma3.5}, we also have
\begin{equation}\label{3.14}
M_{r,a}^+=\inf_{u\in D_a^{rad}\setminus\{0\}}\min_{0<t\leq s_u}I(t\star u)
\end{equation}
and
\begin{equation}\label{3.15}
M_{r,a}^-=\inf_{u\in D_a^{rad}\setminus\{0\}}\max_{s_u<t\leq t_u}I(t\star u).
\end{equation}
\begin{lemma}\label{lemma3.6}
Let $a$ and $\tau$ be positive and satisfy the inequality \eqref{1.12}, then we conclude that
$$ M_a^+=M_{r,a}^+ \text{ and } M_a^-=M_{r,a}^-. $$
\end{lemma}
\begin{proof}
Since $\mathcal{P}_a^+\cap X^{rad}\subset\mathcal{P}_a^+$ and $\mathcal{P}_a^-\cap X^{rad}\subset\mathcal{P}_a^-$, it is trivial that $$M_a^+\leq M_{r,a}^+ \text{~ and~ } M_a^-\leq M_{r,a}^-.$$
Moreover, by Lemma \ref{lemma3.3}, for each $u\in D_a\setminus\{0\}$, there exist positive numbers $s_u$ and $t_u$ with $s_u<t_u$ such that
$$s_u\star u\in \mathcal{P}_a^+ ~~\text{and}~~ t_u\star u\in \mathcal{P}_a^-.$$
Let $u^*\in D_a\cap H_r^1(\R^N)$ be the Schwarz rearrangement of $|u|$. Then
\[ I(t\star u^*)\leq I(t\star u)~~ \text{for all}~~ t\in\R^+. \]
Recalling the definition of $\Psi'_u(t)$ and $\Psi''_u(t)$,   we deduce that
$$ \Psi'_{u^*}(t)\leq \Psi'_u(t) \text{ and }\Psi''_{u^*}(t)\leq \Psi''_u(t)\text{ for all }t\in\R^+, $$
which indicates that $0<s_u\leq s_{u^*}<t_{u^*}\leq t_u$. Therefore, we conclude that
$$ \min_{0<t\leq s_{u^*}}I(t\star u^*)\leq\min_{0<t\leq s_u}I(t\star u)$$
and
$$ \max_{s_{u^*}<t\leq t_{u^*}}I(t\star u^*)\leq\max_{s_u<t\leq t_u}I(t\star u),$$
which implies
$$ \inf_{u\in D_a^{rad}\setminus\{0\}}\min_{0<t\leq s_{u^*}}I(t\star u^*)\leq \inf_{u\in D_a\setminus\{0\}}\min_{0<t\leq s_u}I(t\star u)$$
and
$$ \inf_{u\in D_a^{rad}\setminus\{0\}}\max_{s_{u^*}<t\leq t_{u^*}}I(t\star u^*)\leq \inf_{u\in D_a\setminus\{0\}} \max_{s_u<t\leq t_u}I(t\star u).$$
Combining  \eqref{3.9}, \eqref{3.10}, \eqref{3.14} and \eqref{3.15}, we can have that
$$M_a^+\geq M_{r,a}^+~~\text{and}~~M_a^-\geq M_{r,a}^-.$$
Thus
$$ M_a^+=M_{r,a}^+ \text{ and } M_a^-=M_{r,a}^-.$$
\end{proof}
\begin{lemma}\label{lemma3.7}
Let $ N\geq 1$, $2<q<2+\frac{4}{N}<4+\frac{4}{N}<p<2\cdot2^*$, and let $a$ and $\tau$ be positive and satisfy the inequality \eqref{1.12}. Then $M_a^+$ can be achieved by $u^+$ and $M_a^-$ can be achieved by $u^-$. Moreover, $u^+$ and $u^-$ are radially symmetric.
\end{lemma}
\begin{proof}
Let $\{u_n^+\}\subset \mathcal{P}_a^+\subset A_{R_0}$ be a minimizing sequence for $M_a^+$ and $\{u_n^-\}\subset \mathcal{P}_a^-$ be a minimizing sequence for $M_a^-$, which, combining  \eqref{3.8} and \eqref{3.10}, implies that
$$ I(u_n^+)\rightarrow M_a^+= m(a,\tau) \text{ and } I(u_n^-)\rightarrow M_a^-,\text{  as }n\rightarrow\infty.$$
Consider the Schwarz rearrangement of $|u_n^\pm|$. After passing to a subsequence, we obtain new minimizing sequences $\{u_n^\pm\}$ that are nonnegative, radially symmetric, and non-increasing in $r=|x|$. Since $u_n^\pm\in\mathcal{P}_a^\pm$, \eqref{2.4} implies that
\begin{align*}
I(u_n^\pm)&=I(u_n^\pm)-\frac{2}{N(p-2)}P(u_n^\pm)\\
&=\left(\frac{1}{2}-\frac{2}{N(p-2)}\right)|\bi u_n^\pm|_2^2+\left(1-\frac{2(N+2)}{N(p-2)}\right)V(u_n^\pm)-\frac{\tau(p-q)}{q(p-2)}|u_n^\pm|_q^q\\
&\geq \left(\frac{1}{2}-\frac{2}{N(p-2)}\right)|\bi u_n^\pm|_2^2+\left(1-\frac{2(N+2)}{N(p-2)}\right)V(u_n^\pm)\\
&~~\quad-\frac{\tau(p-q)C_{N,q}}{q(p-2)}a^{\frac{4N-q(N-2)}{2(N+2)}}(V(u_n^\pm))^{\frac{q\gamma_q}{N+2}},
\end{align*}
where $2<q<2+\frac{4}{N}<4+\frac{4}{N}<p<2\cdot2^*$.
So $\{u_n^\pm\}$ is bounded in $X$.
By Proposition 1.7.1 in \cite{CT} and Theorem A.I in \cite{BH}, we find a subsequence (still denoted by  $\{u_n^\pm\}$) such that $u_n^\pm \rightharpoonup u^\pm$ weakly in $H^1(\R^N)$, $u_n^\pm\rightarrow u^\pm$ strongly in $L^q(\R^N)$ with $q\in(2,2\cdot2^*)$ and $u_n^\pm\rightarrow u^\pm$ a.e. in $\R^N$.
Since $V(u^\pm) = \frac{1}{4}|\nabla (u^\pm)^2|_2^2$, we also get $(u_n^\pm)^2 \rightharpoonup (u^\pm)^2$ in $D_0^{1,2}(\R^N)$.
Finally, weak lower semi-continuity yields
\begin{equation}\label{3.16}
\begin{aligned}
\left\{
  \begin{array}{ll}
    |\bi u^\pm|_2^2\leq \liminf\limits_{n\rightarrow\infty}|\bi u_n^\pm|_2^2, \\
    V(u^\pm)\leq \liminf\limits_{n\rightarrow\infty}V(u_n^\pm), \\
    |u^\pm|_2^2\leq \liminf\limits_{n\rightarrow\infty}| u_n^\pm|_2^2, \\
    |u^\pm|_p^p  =  \liminf\limits_{n\rightarrow\infty}| u_n^\pm|_p^p, \\
    |u^\pm|_q^q  =  \liminf\limits_{n\rightarrow\infty}| u_n^\pm|_q^q.
  \end{array}
\right.
\end{aligned}
\end{equation}
We first claim that $u^\pm\neq 0$. Indeed, if it is false, by $|u_n^\pm|_p^p=|u_n^\pm|_q^q=o_n(1)$ and $P(u_n^\pm)=0$, we obtain that
$$ |\bi u_n^\pm|_2^2+(N+2)V(u_n^\pm)=o_n(1),$$
which implies that
\begin{equation}\label{3.17}
|\bi u_n^\pm|_2^2=o_n(1)\text{ and } V(u_n^\pm)=o_n(1).
\end{equation}
So $M_a^\pm =\lim\limits_{n\to +\infty} I(u_n^\pm)=0$,
which contradicts to $M_a^\pm \neq 0$.
 Therefore, $u^\pm\neq 0$.

Next, we show that  $u^+ \in A_{R_0}$. Since $\{u_n^+\} \subset A_{R_0}$, \eqref{3.16} implies that
$$
|u^+|_2^2 \leq \liminf_{n \to \infty} |u_n^+|_2^2 \leq a
$$
and
$$
V(u^+) \leq \liminf_{n \to \infty} V(u_n^+) \leq R_0.
$$
Thus, to establish that  $u^+ \in A_{R_0}$, it suffices to show that  $V(u^+) \neq R_0$. Assume that $V(u^+) = R_0$. Then \eqref{3.1} yields
\begin{equation}\label{3.19}
\begin{aligned}
I(u^+) &= \frac{1}{2}|\nabla u^+|_2^2 + V(u^+) - \frac{1}{p}|u^+|_p^p - \frac{\tau}{q}|u^+|_q^q \\
&= \frac{1}{2}|\nabla u^+|_2^2 + g(R_0) \\
&= \frac{1}{2}|\nabla u^+|_2^2 \\
&\geq 0,
\end{aligned}
\end{equation}
and \eqref{3.16} implies that
\begin{equation}\label{3.20}
\begin{aligned}
I(u^+)&=\frac{1}{2}|\bi u^+|_2^2+V(u^+)-\frac{1}{p}|u^+|_p^p-\frac{\tau}{q}|u^+|_q^q\\
&\leq \liminf_{n\rightarrow\infty}(\frac{1}{2}|\bi u_n^+|_2^2+V(u_n^+))-\liminf_{n\rightarrow\infty}(\frac{1}{p}|u_n^+|_p^p+\frac{\tau}{q}|u_n^+|_q^q)\\
&=\liminf_{n\rightarrow\infty}I(u_n^+)\\
&=m(a,\tau)<0,
\end{aligned}
\end{equation}
a contradiction. Thus $u^+\in A_{R_0}$. From \eqref{3.16}, \eqref{3.20}  and $u^+\in A_{R_0}$, we obtain
$$
m(a,\tau) \leq I(u^+) \leq \liminf_{n\to\infty} I(u_n^+) = m(a,\tau),
$$
which yields
\begin{equation}\label{3.21}
\lim_{n\to\infty}|\nabla u_n^+|_2^2 = |\nabla u^+|_2^2 ~~ \text{and}~~\lim_{n\to\infty} V(u_n^+) = V(u^+)
\end{equation}
and
\begin{equation}\label{3.22}
I(u^+) = m(a,\tau).
\end{equation}
Combining \eqref{3.8}, \eqref{3.21}, \eqref{3.22} and $u_n^+\in \mathcal{P}_a^+$, we conclude
$$
u^+ \in \mathcal{P}_a^+ ~~ \text{and}~~ I(u^+) = M_a^+,
$$
which implies  that $M_a^+$ is attained by $u^+$.

For the case of $M_a^-$, since $u^- \neq 0$ and $u^- \in D_a\setminus\{0\}$, it follows from Lemma \ref{lemma3.3} that $t_{u^-} > 0$ with $t_{u^-}\star u^- \in \mathcal{P}_a^-$. Using \eqref{1.8}, \eqref{3.16}  and $u_n^- \in \mathcal{P}_a^-$, we have
$$
M_a^- \leq I(t_{u^-}\star u^-) \leq \liminf_{n\to\infty} I(t_{u^-}\star u_n^-) \leq \liminf_{n\to\infty} I(u_n^-) = M_a^-,
$$
which implies that
$$
\lim_{n\to\infty}|\nabla u_n^-|_2^2 = |\nabla u^-|_2^2 ~~ \text{and} ~~ \lim_{n\to\infty} V(u_n^-) = V(u^-),
$$
with $I(u^-) = M_a^-$. Therefore, $u^- \in \mathcal{P}_a^-$ and $M_a^-$ is attained by $u^-$.

We complete  the proof.
\end{proof}

\begin{lemma}\label{lemma3.80}
Let $0 \neq u \in D_a$ satisfy $I(u) < m(a, \tau)$. Then $t_u < 1$.
\end{lemma}

\begin{proof}
Taking  $0 \neq u \in D_a$ with $I(u) < m(a, \tau)$, it follows from Lemma \ref{lemma3.3} that  there exist $s_u < c_u < t_u < d_u$ satisfies the properties of Lemma \ref{lemma3.3}. If  we can  show $d_u \leq 1$, then  $t_u < 1$. Suppose, by contradiction, that $d_u > 1$. Since $\Psi_u(1) = I(u) < m(a, \tau) < 0$, we have that  $c_u > 1$(In fact, if $c_u<1$, then, according to the fact that $ \Psi_u(t)>0$ in $(c_u, d_u)$, we can see that $ \Psi_u(1)>0$, contradicting to $ \Psi_u(1)=I(u)<0$; if $c_u=1$, then $ \Psi_u(1)=0$, which contradicts to $ \Psi_u(1)=I(u)<0$). So, following from Lemma \ref{lemma3.3}, we can have that
\begin{align*}
m(a,\tau)
&> I(u) = \Psi_u(1) \\
&\geq \min_{t \in (0,c_u)} \Psi_u(t) \\
&\geq \min \{ I(t \star u) \mid t > 0,  V(t \star u) < R_0 \} \\
&= I(s_u \star u) \\
&\geq m(a,\tau),
\end{align*}
a contradiction.
\end{proof}

By Lemma~\ref{lemma3.7}, there exists $u^+ \in \mathcal{P}_a^+$ such that $I(u^+) = m(a, \tau) < 0$. For large $t > 0$, it follows that $I(t \star u^+) < 2m(a, \tau)$. Hence, the set
$$
\Gamma := \left\{ \gamma \in \mathcal{C}\big([0, 1], D_a \setminus \{0\}\big) : \gamma(0) \in \mathcal{P}_a^+,\  I(\gamma(1)) \leq 2m(a, \tau) \right\}
$$
is nonempty. We define the minimax value
$$
m^* := \inf_{\gamma \in \Gamma} \sup_{\theta \in [0, 1]} I(\gamma(\theta)).
$$

\begin{lemma}\label{lemma3.90}
It holds that $m^* = \inf\limits_{u \in \mathcal{P}_a^-} I(u) > 0$.
\end{lemma}

\begin{proof}
From \eqref{3.10}, we have that  $\inf\limits_{u \in \mathcal{P}_a^-} I(u) > 0$. For any $\gamma \in \Gamma$, we have $t_{\gamma(0)} > s_{\gamma(0)} = 1$. By Lemma \ref{lemma3.80}, $t_{\gamma(1)} < 1$. Thus it follows from Lemma \ref{lemma3.3} that there is a  $\theta_0 = \theta_0(\gamma) \in (0, 1)$ such that  $t_{\gamma(\theta_0)} = 1$. So
$$
\gamma(\theta_0) \in \mathcal{P}_a^-~~ \text{and}~~\sup_{\theta \in [0,1]} I(\gamma(\theta)) \geq I(\gamma(\theta_0)) \geq \inf_{u \in \mathcal{P}_a^-} I(u).
$$
Hence $m^* \geq \inf\limits_{u \in \mathcal{P}_a^-} I(u)$.

For any  $u \in \mathcal{P}_a^-$, we have that $s_u < t_u = 1$. Choose $\theta_1 > 1$ such that $I(\theta_1 \star u) < 2m(a, \tau)$. Define $\gamma: [0,1] \to D_a \setminus \{0\}$ by
$$
\gamma(\theta) = \big( \theta \theta_1 + (1-\theta) s_u \big) \star u.
$$
Then $\gamma \in \Gamma$ and
$$
I(u) = \sup_{\theta \in [0,1]} I(\gamma(\theta)) \geq m^*.
$$
Thus $\inf\limits_{u \in \mathcal{P}_a^-} I(u) \geq m^*$.

Therefore we conclude $m^* = \inf\limits_{u \in \mathcal{P}_a^-} I(u) > 0$.
\end{proof}

\begin{lemma}\label{lemma3.100}
If $a, \tau > 0$ satisfy  \eqref{1.12}, then
$$
m^* = \inf_{u \in D_a \setminus \{0\}} \max_{t > 0} I(t \star u).
$$
\end{lemma}

\begin{proof}
By Lemma \ref{lemma3.3}, each $u \in D_a \setminus \{0\}$ has a unique $t_u > 0$ maximizing $\Psi_u(t)$ with $t_u \star u \in \mathcal{P}_a^-$, which implies that
$$
\max_{t > 0} I(t \star u) = I(t_u \star u) \geq \inf_{v \in \mathcal{P}_a^-} I(v).
$$
Thus
$$
\inf_{u \in D_a \setminus \{0\}} \max_{t > 0} I(t\star u) \geq \inf_{u \in \mathcal{P}_a^-} I(u)=m^*.
$$

For any $u \in \mathcal{P}_a^-$, we have that $t_u = 1$ and
$$
I(u) = \max_{t > 0} I(t \star u) \geq \inf_{v \in D_a \setminus \{0\}} \max_{t > 0} I(t \star v),
$$
which implies that
\begin{equation}\label{3.2300}
m^*=\inf_{u \in \mathcal{P}_a^-} I(u) \geq \inf_{u \in D_a \setminus \{0\}} \max_{t > 0} I(t \star u).
\end{equation}

Therefore, we have that
$$
m^* = \inf_{u \in D_a \setminus \{0\}} \max_{t > 0} I(t \star u).
$$
\end{proof}

\begin{lemma}\label{lemma3.800}
Assume that  $N \geq 1$, $2 < q < 2 + \frac{4}{N} < 4 + \frac{4}{N} < p < 2 \cdot 2^*$, and   $a, \tau>0 $  satisfy the inequality \eqref{1.12}. If $u^+ \in A_{R_0}$ and $I(u^+) = m(a,\tau)$, then $u^+$ is a weak solution of  \eqref{1.1}, i.e,  there exists $\lambda^+ \in \mathbb{R}$ such that
$$ \langle I'(u^+), \phi\rangle + \lambda^+ \int_{\mathbb{R}^N} u^+ \phi  dx = 0, \quad \forall \phi \in \mathcal{C}_0^\infty(\mathbb{R}^N). $$
\end{lemma}

\begin{proof}
Firstly, fix an arbitrary function $\phi \in \mathcal{C}_0^\infty(\mathbb{R}^N)$. By Lemma \ref{lemma3.7}, we have that  $|u^+|_2^2 \leq a$ and $u^+ \neq 0$. Thus, we may select a function $\omega \in \mathcal{C}_0^\infty(\mathbb{R}^N)$ satisfying
\begin{equation}\label{3.230}
\int_{\mathbb{R}^N} u^+ \omega  dx \neq 0.
\end{equation}

The proof is divided into two cases.

\textbf{Case 1:} Suppose that  $u^+ \in \mathcal{A} := \{ u \mid |u|_2^2 = a, V(u) < R_0 \}$. Introduce $J(u) := |u|_2^2 - a$ and define the function
\begin{equation}\label{3.240}
\begin{aligned}
g(\eta,\sigma) &:= J(u^+ + \eta \phi + \sigma \omega) \\
&= \int_{\mathbb{R}^N} |u^+ + \eta \phi + \sigma \omega|^2  dx - a, \quad (\eta, \sigma \in \mathbb{R}).
\end{aligned}
\end{equation}
Observe that
\begin{equation}\label{3.250}
g(0,0) = \int_{\mathbb{R}^N} |u^+|^2  dx - a = 0.
\end{equation}
Since $g$ is $\mathcal{C}^1$, we compute the partial derivatives:
\begin{equation}\label{3.260}
\frac{\partial g}{\partial \eta}(\eta,\sigma) = 2 \int_{\mathbb{R}^N} (u^+ + \eta \phi + \sigma \omega) \phi  dx,
\end{equation}
\begin{equation}\label{3.270}
\frac{\partial g}{\partial \sigma}(\eta,\sigma) = 2 \int_{\mathbb{R}^N} (u^+ + \eta \phi + \sigma \omega) \omega  dx.
\end{equation}
From \eqref{3.230}, we deduce that
\begin{equation}\label{3.280}
\frac{\partial g}{\partial \sigma}(0,0) \neq 0.
\end{equation}
By the implicit function theorem, there exists a $\mathcal{C}^1$ function $\psi : \mathbb{R} \to \mathbb{R}$ such that
\begin{equation}\label{3.290}
\psi(0) = 0
\end{equation}
and
\begin{equation}\label{3.300}
g(\eta, \psi(\eta)) = 0, ~~\eta\in [-\eta_0, \eta_0]
\end{equation}
for a sufficiently small positive constant $\eta_0$. Differentiating with respect to $\eta$ implies
$$
\frac{\partial g}{\partial \eta}(\eta, \psi(\eta)) + \frac{\partial g}{\partial \sigma}(\eta, \psi(\eta)) \psi'(\eta) = 0,
$$
which, combining  \eqref{3.260} and \eqref{3.270}, implies that
\begin{equation}\label{3.310}
\psi'(0) = - \frac{\int_{\mathbb{R}^N} u^+ \phi  dx}{\int_{\mathbb{R}^N} u^+ \omega  dx}.
\end{equation}

Letting
\begin{equation}\label{3.321}
h(\eta) := \eta \phi + \psi(\eta) \omega, \quad (|\eta| \leq \eta_0)
\end{equation}
and setting
$$
i(\eta) := I(u^+ + h(\eta)),
$$
it follows from  \eqref{3.300} and \eqref{3.321}  that $J(u^+ + h(\eta)) = 0$ and $V(u^+ + h(\eta)) < R_0$. Thus, $u^+ + h(\eta)\in \mathcal{A} \subset A_{R_0}$, which tells us that
$$
i(\eta) = I(u^+ + h(\eta)) \geq \inf_{u \in A_{R_0}} I(u) = I(u^+) = i(0).
$$
So we can see that  the $\mathcal{C}^1$ function $i(\eta)$ attains its minimum at $\eta = 0$. Then
\begin{equation}\label{3.320}
\begin{aligned}
0 &= i'(0) \\
&= \int_{\mathbb{R}^N} \left[ \bi u^+ (\bi \phi + \psi'(0) \bi \omega) + \bi (u^+)^2 (\bi (u^+ \phi) + \psi'(0) \bi (u^+ \omega)) \right] dx \\
&\quad - \int_{\mathbb{R}^N} \left[ |u^+|^{p-1} \phi + \psi'(0) |u^+|^{p-1} \omega + \tau |u^+|^{q-1} \phi + \tau \psi'(0) |u^+|^{q-1} \omega \right] dx,
\end{aligned}
\end{equation}

Letting
$$
\lambda^+ = - \frac{\int_{\mathbb{R}^N} \left[ \bi u^+ \bi \omega + \bi (u^+)^2 \bi (u^+ \omega) - |u^+|^{p-1} \omega - \tau |u^+|^{q-1} \omega \right] dx}{\int_{\mathbb{R}^N} u^+ \omega  dx},
$$
it follows from  \eqref{3.310} that
$$
\langle I'(u^+),  \phi\rangle + \lambda^+ \int_{\mathbb{R}^N} u^+ \phi  dx = 0.
$$

\textbf{Case 2:} Suppose that  $u^+ \in \mathcal{B} := \{ u \mid |u|_2^2 < a, V(u) < R_0 \}$. Then there exists $\eta_0 > 0$ sufficiently small such that,  for all $\eta \in [-\eta_0, \eta_0]$,
$$
\int_{\mathbb{R}^N} |u^+ + \eta \phi|^2  dx < a ~~ \text{and} ~~ V(u^+ + \eta \phi) < R_0.
$$
Letting
$$
f(\eta) := I(u^+ + \eta \phi),
$$
then it is easy to see that $ u^+ + \eta \phi\in \mathcal{B} \subset A_{R_0}$ and
$$
f(\eta) \geq \inf_{u \in A_{R_0}} I(u) = I(u^+) = f(0).
$$
Thus $f$ has a minimum at $\eta = 0$. So
\begin{equation}\label{3.330}
0 = f'(0) = \int_{\mathbb{R}^N} \left[ \bi u^+ \bi \phi + \bi (u^+)^2 \bi (u^+ \phi) - |u^+|^{p-1} \phi - \tau |u^+|^{q-1} \phi \right] dx.
\end{equation}
Therefore, we have that
$$
\langle I'(u^+),  \phi \rangle+ \lambda^+ \int_{\mathbb{R}^N} u^+ \phi  dx = 0,
$$
where $\lambda^+ = 0$.

The proof is completed.
\end{proof}

\begin{lemma}\label{lemma3.8}
Let $N \geq 1$, $2 < q < 2 + \frac{4}{N} < 4 + \frac{4}{N} < p < 2 \cdot 2^*$, and suppose that  $a, \tau > 0$ satisfy \eqref{1.12}. If $u^- \in \mathcal{P}_a^-$ satisfies $I(u^-) = M_a^-$, then $u^-$ is a weak solution of \eqref{1.1}. Namely, there exists $\lambda^- \in \mathbb{R}$ such that
$$
\langle I'(u^-), \phi \rangle + \lambda^- \int_{\mathbb{R}^N} u^- \phi  dx = 0 ~~ \text{for all} ~~ \phi \in C_0^\infty(\mathbb{R}^N).
$$
\end{lemma}

\begin{proof}
Since $M_a^- > 0$, we have that $u^- \neq 0$. Setting  $b := \|u^-\|_2^2$, it is easy to see  that $0 < b \leq a$. Define
$$
\tilde{M}_{b}^- := \inf_{v \in \mathcal{P}^- \cap \mathcal{S}_b} I(v).
$$
Clearly $\tilde{M}_{b}^- \geq M_b^-$. Combining   $u^- \in \mathcal{P}^- \cap \mathcal{S}_b$, $b \leq a$, and the monotonicity of $c \mapsto M_c^-$, we obtain
$$
M_a^- = I(u^-) \geq \tilde{M}_{b}^- \geq M_{b}^- \geq M_a^-.
$$
Thus $I(u^-) = \tilde{M}_{b}^-$. Using a similar method  as that  in \cite[Lemma 2.5]{ZL}, there exists $\lambda^- \in \mathbb{R}$ such that
$$
\langle I'(u^-), \phi\rangle + \lambda^- \int_{\mathbb{R}^N} u^- \phi  dx = 0 \quad \text{for all} \quad \phi \in C_0^\infty(\mathbb{R}^N).
$$
\end{proof}


\subsection{Properties of Lagrange multiplier}
In this subsection, by studying the properties of $\lambda^\pm$, we conclude that the minimizers $u^+$ and $u^-$ belong to $\mathcal{S}_a$.

\begin{lemma}\label{lemma3.9}
Let $N \geq 1$, $2 < q < 2 + \frac{4}{N} < 4 + \frac{4}{N} < p < 2 \cdot 2^*$, and  $\lambda^\pm$ be given by Lemmas \ref{lemma3.800} and \ref{lemma3.8}. Then $\lambda^\pm \geq 0$. Moreover,  if $\lambda^\pm \neq  0$, then  $|u^\pm|_2^2=a$.
\end{lemma}
\begin{proof}
Given $0 < |u^\pm|_2^2 \leq a$, choose $\epsilon > 0$ small enough  such that $(1+r)u^\pm \in D_a$ for all $r \in \gamma$, where
$$
\gamma :=
\begin{cases}
(-\epsilon, \epsilon) & \text{if } |u^\pm|_2^2 < a, \\
(-\epsilon, 0] & \text{if } |u^\pm|_2^2 = a.
\end{cases}
$$
For any $r \in \gamma$, consider the scaled energy
\begin{align*}
I(t \star (1+r)u^\pm)& = \frac{1}{2}|\bi u^\pm|_2^2(1+r)^2t^2 + V(u^\pm)(1+r)^4t^{N+2} \\
&-\frac{1}{p}|u^\pm|_p^p(1+r)^p t^{\frac{N(p-2)}{2}} - \frac{\tau}{q}|u^\pm|_q^q(1+r)^q t^{\frac{N(q-2)}{2}}.
\end{align*}
By Lemma \ref{lemma3.3}, there exist a unique $s_{u^+}(r) > 0$ and a unique $t_{u^-}(r) > 0$ such that
$$
I(s_{u^+}(r) \star ((1+r)u^+)) = \min_{0 < t \leq s_{u^+}} I(t \star ((1+r)u^+))
$$
and
$$
I(t_{u^-}(r) \star ((1+r)u^-)) = \max_{s_{u^-} < t \leq t_{u^-}} I(t \star ((1+r)u^-)),
$$
where $s_{u^+}(r)$ and $t_{u^-}(r)$ solve the following equations
\begin{equation}\label{3.23}
\begin{aligned}
&|\bi u^+|_2^2(1+r)^2 + (N+2)V(u^+)(1+r)^4 s^N_{u^+}(r) \\
&- \frac{N(p-2)}{2p}|u^+|_p^p(1+r)^p s^{\frac{N(p-2)-4}{2}}_{u^+}(r) - \frac{N\tau(q-2)}{2q}|u^+|_q^q(1+r)^q s^{\frac{N(q-2)-4}{2}} _{u^+}(r)= 0
\end{aligned}
\end{equation}
and
\begin{equation}\label{3.24}
\begin{aligned}
&|\bi u^-|_2^2(1+r)^2 + (N+2)V(u^-)(1+r)^4 t^N_{u^-}(r) \\
&- \frac{N(p-2)}{2p}|u^-|_p^p(1+r)^p t^{\frac{N(p-2)-4}{2}}_{u^-}(r) - \frac{N\tau(q-2)}{2q}|u^-|_q^q(1+r)^q t^{\frac{N(q-2)-4}{2}}_{u^-}(r) = 0.
\end{aligned}
\end{equation}
The implicit function theorem and Lemma \ref{lemma3.3} imply
\begin{equation}\label{3.25}
s_{u^+}(r),\, t_{u^-}(r) \in \mathcal{C}^2(\mathbb{R}), ~ s_{u^+}(0) = 1, ~ t_{u^-}(0) = 1.
\end{equation}

Define the functions
$$
\rho_+(r) := I(s_{u^+}(r) \star ((1+r)u^+)), ~ r \in \gamma
$$
and
$$
\rho_-(r) := I(t_{u^-}(r) \star ((1+r)u^-)), ~ r \in \gamma.
$$
Since $s_{u^+}(r) \star ((1+r)u^+) \in \mathcal{P}_a^+$, $t_{u^-}(r) \star ((1+r)u^-) \in \mathcal{P}_a^-$, and $I(u^\pm) = M_a^\pm$, we have that
$$
\rho_+(r) \geq M_a^+ = I(u^+) = \rho_+(0), ~~ \rho_-(r) \geq M_a^- = I(u^-) = \rho_-(0).
$$
Thus,
\begin{equation}\label{3.26}
\rho'_\pm(0)
\begin{cases}
= 0 & \text{if } |u^\pm|_2^2 < a \text{ (interior point)}, \\
\leq 0 & \text{if } |u^\pm|_2^2 = a \text{ (boundary point)},
\end{cases}
\end{equation}
where $\rho'_\pm(0)$ denotes the left derivative when $|u^\pm|_2^2 = a$.

For $\rho_+(r)$, based on  \eqref{3.23} and the expression
\begin{align*}
\rho_+(r) =& \frac{1}{2}|\bi u^+|_2^2(1+r)^2 s^2_{u^+}(r) + V(u^+)(1+r)^4 s^{N+2}_{u^+}(r) \\
&- \frac{1}{p}|u^+|_p^p(1+r)^p s^{\frac{N(p-2)}{2}}_{u^+}(r) - \frac{\tau}{q}|u^+|_q^q(1+r)^q s^{\frac{N(q-2)}{2}}_{u^+}(r),
\end{align*}
direct computation implies that
\begin{align}\label{3.27}
\rho'_+(r) &= |\bi u^+|_2^2(1+r) s^2_{u^+}(r) + 4V(u^+)(1+r)^3 s^{N+2}_{u^+}(r) \\ \notag
&\quad - |u^+|_p^p(1+r)^{p-1} s^{\frac{N(p-2)}{2}}_{u^+}(r) - \tau |u^+|_q^q (1+r)^{q-1} s^{\frac{N(q-2)}{2}}_{u^+}(r) \\ \notag
&\quad + s_{u^+}(r)s'_{u^+}(r) \Bigg[ |\bi u^+|_2^2(1+r)^2 + (N+2)V(u^+)(1+r)^4 s^N _{u^+}(r)\\ \notag
&\quad - \frac{N(p-2)}{2p}|u^+|_p^p(1+r)^p s^{\frac{N(p-2)-4}{2}}_{u^+}(r) - \frac{\tau N(q-2)}{2q}|u^+|_q^q(1+r)^q s^{\frac{N(q-2)-4}{2}}_{u^+}(r) \Bigg] \\ \notag
&= |\bi u^+|_2^2(1+r) s^2_{u^+}(r) + 4V(u^+)(1+r)^3 s^{N+2}_{u^+}(r) \\ \notag
&\quad - |u^+|_p^p(1+r)^{p-1} s^{\frac{N(p-2)}{2}}_{u^+}(r) - \tau |u^+|_q^q (1+r)^{q-1} s^{\frac{N(q-2)}{2}}_{u^+}(r).
\end{align}

It follows from \eqref{3.25} and \eqref{3.27} that
\begin{equation}\label{3.28}
\rho'_+(0) = |\bi u^+|_2^2 + 4V(u^+) - |u^+|_p^p - \tau |u^+|_q^q.
\end{equation}
By Lemma \ref{lemma3.800}, we have that
$$
|\bi u^+|_2^2 + 4V(u^+) + \lambda^+|u^+|_2^2 = |u^+|_p^p + \tau |u^+|_q^q,
$$
which,together  with \eqref{3.28},  yields that
\begin{equation}\label{3.29}
\lambda^+ = -\frac{\rho'_+(0)}{|u^+|_2^2}.
\end{equation}
Similarly, we can  deduce that
\begin{equation}\label{3.30}
\lambda^- = -\frac{\rho'_-(0)}{|u^-|_2^2}.
\end{equation}
Following from \eqref{3.26}, \eqref{3.29}, and \eqref{3.30}, we conclude that $\lambda^\pm \geq 0$. When  $\lambda^\pm \neq  0$, we can conclude that $\rho'_\pm(0)\neq 0$,
 which, together with \eqref{3.26}, implies that $|u^\pm|_2^2 =a$.
\end{proof}

%

\begin{lemma}\label{lemma3.11}
Assume that $1 \leq N \leq 4$, and  $\lambda^\pm$ and $u^\pm$  are given by Lemmas \ref{lemma3.800} and \ref{lemma3.8}.  We have that $\lambda^\pm > 0$.
\end{lemma}

\begin{proof}
It follows from Lemma \ref{lemma3.9} that  $\lambda^\pm \geq 0$.  Suppose, by   contradiction,  that $\lambda^\pm = 0$. Then $u^\pm \geq 0$ with $u^\pm \not\equiv 0$ solves
$$
-\Delta u - \Delta(|u|^2)u = |u|^{p-2}u + \tau |u|^{q-2}u \quad \text{in} \quad \mathbb{R}^N.
$$
Define $u^\pm = \varphi(v^\pm)$, where $\varphi$ is the inverse of the function
$$
v(s) = \int_0^s \sqrt{1 + 2t^2}  dt = \frac{1}{2} s \sqrt{1 + 2s^2} + \frac{\sqrt{2}}{4} \ln \left( \sqrt{2} s + \sqrt{1 + 2s^2} \right).
$$
Then $v^\pm \in H^1(\mathbb{R}^N)$ are Schwarz symmetric and satisfy
$$
-\Delta v^\pm = \varphi(v^\pm)^{p-1} \varphi'(v^\pm) + \tau \varphi(v^\pm)^{q-1} \varphi'(v^\pm) \geq 0 ~ \text{in} ~ \mathbb{R}^N,
$$
which is a contradiction (see \cite[Lemma A.2]{IN}) when  $1 \leq N \leq 4$.  Therefore  $\lambda^\pm >0$.
\end{proof}

\begin{proof}[\bf Proof of Theorem \ref{Th1}:]
Let $a, \tau > 0$ satisfy \eqref{1.12}. By Lemma \ref{lemma3.7}, we can see that $M_a^\pm$ are attained by radially symmetric functions $u^\pm \in \mathcal{P}_a^\pm$.

(1) Lemmas \ref{lemma3.800}, \ref{lemma3.9} and   \ref{lemma3.11} imply that $u^+$ is a weak solution of \eqref{1.1}, $\lambda^+ > 0$ and $u^+ \in \mathcal{S}_a$. Also, we have $I(u^+)=\inf\limits_{v\in\mathcal{P}_a}I(v)$. Thus $u^+$ is a normalized ground state solution.

(2) Similarly, it follows from Lemmas \ref{lemma3.8}, \ref{lemma3.9} and   \ref{lemma3.11}   that $u^-$ is a weak solution of \eqref{1.1} with  $\lambda^- > 0$ and $u^- \in \mathcal{S}_a$. Furthermore, by Lemmas \ref{lemma3.80}, \ref{lemma3.90} and \ref{lemma3.100}, we can see that there exists a  mountain pass geometry structure for $I|_{\mathcal{P}_a^-}$ at level $m^*$. Thus we obtain that  $u^- \in \mathcal{P}_a^-$ satisfies $I(u^-) = M_a^- = m^*$, which implies that $u^-$ is a normalized mountain pass solution of  \eqref{1.1}.

We complete the proof.
\end{proof}

\end{document}